\documentclass{amsproc}
\usepackage{amsfonts}

\theoremstyle{plain}

\newtheorem{corollary}{Corollary}

\newtheorem{lemma}{Lemma}

\newtheorem{proposition}{Proposition}
\newtheorem{remark}{Remark}

\newtheorem{theorem}{Theorem}
\numberwithin{equation}{section}
\input{tcilatex}

\begin{document}
\title[Spectral Radius Inequalities]{Spectral Radius Inequalities for
Functions of Operators Defined by Power Series}
\author{S.S. Dragomir$^{1,2}$}
\address{$^{1}$Mathematics, School of Engineering \& Science\\
Victoria University, PO Box 14428\\
Melbourne City, MC 8001, Australia.}
\email{sever.dragomir@vu.edu.au}
\urladdr{http://rgmia.org/dragomir}
\address{$^{2}$School of Computational \& Applied Mathematics, University of
the Witwatersrand, Private Bag 3, Johannesburg 2050, South Africa}
\subjclass{47A63; 47A99.}
\keywords{Bounded linear operators, Functions of operators, Numerical
radius, Power series}

\begin{abstract}
By the help of power series $f\left( z\right) =\sum_{n=0}^{\infty
}a_{n}z^{n} $ we can naturally construct another power series that has as
coefficients the absolute values of the coefficients of $f$, namely $%
f_{a}\left( z\right) :=\sum_{n=0}^{\infty }\left\vert a_{n}\right\vert
z^{n}. $ Utilising these functions we show among others that%
\begin{equation*}
r\left[ f\left( T\right) \right] \leq f_{a}\left[ r\left( T\right) \right]
\end{equation*}%
where $r\left( T\right) $ denotes the spectral radius of the bounded linear
operator $T$ on a complex Hilbert space while $\left\Vert T\right\Vert $ is
its norm. When we have $A$ and $B$ two commuting operators, then%
\begin{equation*}
r^{2}\left[ f\left( AB\right) \right] \leq f_{a}\left( r^{2}\left( A\right)
\right) f_{a}\left( r^{2}\left( B\right) \right)
\end{equation*}%
and%
\begin{equation*}
r\left[ f\left( AB\right) \right] \leq \frac{1}{2}\left[ f_{a}\left(
\left\Vert AB\right\Vert \right) +f_{a}\left( \left\Vert A^{2}\right\Vert
^{1/2}\left\Vert B^{2}\right\Vert ^{1/2}\right) \right] .
\end{equation*}
\end{abstract}

\maketitle

\section{Introduction}

Let $B(H$) denote the algebra of all bounded linear operators on a complex
Hilbert space $H$. For an operator $T\in B(H)$, let $r(T)$ and $\left\Vert
T\right\Vert $ denote the \textit{spectral radius} and the usual \textit{%
operator norm} of \textit{A}, respectively. It is well known that for every $%
T\in B(H$), we have the fundamental inequality 
\begin{equation}
r(T)\leq \left\Vert T\right\Vert  \label{e.0.1}
\end{equation}
and that equality holds in the inequality (\ref{e.0.1}) if $T$ is normal.

In addition to the inequality (\ref{e.0.1}), the most important properties
of the spectral radius are the \textit{spectral radius formula}

\begin{equation}
r(T)=\lim_{n\rightarrow \infty }\left\Vert T^{n}\right\Vert ^{1/n},
\label{e.0.2}
\end{equation}%
a special case of the \textit{spectral mapping theorem}, which asserts that

\begin{equation}
r(T^{m})=r^{m}(T)  \label{e.0.3}
\end{equation}
for every natural number $m$, and the \textit{commutativity property}, which
asserts that

\begin{equation}
r(AB)=r(BA)\text{ for every }A,B\in B(H).  \label{e.0.4}
\end{equation}%
It follows from the spectral radius formula (2) that if $A,B\in B(H)$ are 
\textit{commutative} then the following \textit{subadditivity}

\begin{equation}
r(A+B)\leq r(A)+r(B)  \label{e.0.5}
\end{equation}%
and \textit{submultiplicativity} 
\begin{equation}
r(AB)\leq r(A)r(B)  \label{e.0.6}
\end{equation}%
properties hold.

For additional properties of the spectral radius, the reader is referred to
the classical book \cite{H} and the papers \cite{LD}-\cite{Y}.

There are simple examples, see for instance \cite{Ki}, showing that the
properties (\ref{e.0.5}) and (\ref{e.0.6}) are not true for non-commutative
operators $A$ and $B.$

In \cite{Ki} the author has proved the following inequality%
\begin{align}
& r\left( AB\pm BA\right)  \label{e.0.7} \\
& \leq \frac{1}{2}\left( \left\Vert AB\right\Vert +\left\Vert BA\right\Vert +%
\sqrt{\left( \left\Vert AB\right\Vert -\left\Vert BA\right\Vert \right)
^{2}+4\left\Vert A^{2}\right\Vert \left\Vert B^{2}\right\Vert }\right) 
\notag
\end{align}%
for any $A,B\in B(H).$

If $A$ and $B$ are commutative, then from (\ref{e.0.7}) we get%
\begin{equation}
r\left( AB\right) \leq \frac{1}{2}\left( \left\Vert AB\right\Vert
+\left\Vert A^{2}\right\Vert ^{1/2}\left\Vert B^{2}\right\Vert ^{1/2}\right)
,  \label{e.0.8}
\end{equation}%
which is of interest in itself and also has some nice applications for
functions of operators as follows.

In the same paper \cite{Ki}, the author also provided the inequality below%
\begin{eqnarray}
r\left( AB\pm BA\right) &\leq &\left\Vert AB\right\Vert +\min \left\{
\left\Vert A\right\Vert ^{1/2}\left\Vert AB^{2}\right\Vert ^{1/2},\left\Vert
A^{2}B\right\Vert ^{1/2}\left\Vert B\right\Vert ^{1/2}\right\}  \label{e.0.9}
\\
&\leq &\left\Vert AB\right\Vert +\left\{ 
\begin{array}{ll}
\left\Vert A\right\Vert ^{1/2}\left\Vert B\right\Vert ^{1/2}\left\Vert
AB\right\Vert ^{1/2}, &  \\ 
&  \\ 
\min \left\{ \left\Vert A\right\Vert \left\Vert B^{2}\right\Vert
^{1/2},\left\Vert A^{2}\right\Vert ^{1/2}\left\Vert B\right\Vert \right\} ,
& 
\end{array}%
\right.  \notag
\end{eqnarray}%
which produces in the case of commutative $A$ and $B$ the string of
inequalities that are also useful in what follows:%
\begin{align}
r\left( AB\right) & \leq \frac{1}{2}\left[ \left\Vert AB\right\Vert +\min
\left\{ \left\Vert A\right\Vert ^{1/2}\left\Vert AB^{2}\right\Vert
^{1/2},\left\Vert A^{2}B\right\Vert ^{1/2}\left\Vert B\right\Vert
^{1/2}\right\} \right]  \label{e.0.10} \\
& \leq \frac{1}{2}\left\Vert AB\right\Vert +\frac{1}{2}\times \left\{ 
\begin{array}{ll}
\left\Vert A\right\Vert ^{1/2}\left\Vert B\right\Vert ^{1/2}\left\Vert
AB\right\Vert ^{1/2}, &  \\ 
&  \\ 
\min \left\{ \left\Vert A\right\Vert \left\Vert B^{2}\right\Vert
^{1/2},\left\Vert A^{2}\right\Vert ^{1/2}\left\Vert B\right\Vert \right\} .
& 
\end{array}%
\right.  \notag
\end{align}

Motivated by the above results we establish in this paper some inequalities
for the spectral radius of functions of operators defined by power series,
which incorporate many fundamental functions of interest such as the
exponential function, some trigonometric functions, the functions $f\left(
z\right) =\left( 1\pm z\right) ^{-1},$ $g\left( z\right) =\log \left( 1\pm
z\right) ^{-1}$ and others. Some examples of interest are also provided.

\section{Inequalities for One Operator}

We start with the following lemmas.

\begin{lemma}
\label{l.2.1}Let $\left( V_{j}\right) _{j\in \mathbb{N}}$ be a sequence of
bounded linear operators such that $V_{j}V_{k}=V_{k}V_{j}$ for any $j,k\in 
\mathbb{N}$. The for any $m\in \mathbb{N}$, $m\geq 1$ we have%
\begin{equation}
r\left( \sum_{j=0}^{m}V_{j}\right) \leq \sum_{j=0}^{m}r\left( V_{j}\right) .
\label{e.2.1}
\end{equation}
\end{lemma}

\begin{proof}
By induction over $m.$

If $m=1,$ the inequality follows by (\ref{e.0.5}).

Assume that (\ref{e.2.1}) is true for $m>1.$ Since the operators $%
\sum_{j=0}^{m}V_{j}$ and $V_{m+1}$ are commutative, then by (\ref{e.0.5}) we
also have%
\begin{align*}
r\left( \sum_{j=0}^{m+1}V_{j}\right) & =r\left(
\sum_{j=0}^{m}V_{j}+V_{m+1}\right) \leq r\left( \sum_{j=0}^{m}V_{j}\right)
+r\left( V_{m+1}\right) \\
& \leq \sum_{j=0}^{m}r\left( V_{j}\right) +r\left( V_{m+1}\right)
=\sum_{j=0}^{m+1}r\left( V_{j}\right) ,
\end{align*}%
where for the last inequality we used the induction hypothesis.

This proves the inequality (\ref{e.2.1}) for any $m\geq 1.$
\end{proof}

\begin{lemma}
\label{l.2.2}If $V,S\in B(H)$ are \textit{commutative then the following
continuity property holds}%
\begin{equation}
\left\vert r\left( V\right) -r\left( S\right) \right\vert \leq r\left(
V-S\right) .  \label{e.2.2}
\end{equation}
\end{lemma}

\begin{proof}
Since $V-S$ and $S$ are commutative, then by (\ref{e.0.5}) we have%
\begin{equation*}
r\left( V\right) =r\left( V-S+S\right) \leq r\left( V-S\right) +r\left(
S\right)
\end{equation*}%
giving that 
\begin{equation}
r\left( V\right) -r\left( S\right) \leq r\left( V-S\right) .  \label{e.2.3}
\end{equation}%
From (\ref{e.0.2}) we have that $r(-T)=r\left( T\right) $ for any operator $%
T.$

Since the operators $S-V$ and $V$ also commute, then 
\begin{equation*}
r\left( S\right) \leq r\left( S-V\right) +r\left( V\right)
\end{equation*}%
showing that%
\begin{equation*}
r\left( S\right) -r\left( V\right) \leq r\left( S-V\right) =r\left(
V-S\right)
\end{equation*}%
or, equivalently%
\begin{equation}
-r\left( V-S\right) \leq r\left( V\right) -r\left( S\right) .  \label{e.2.4}
\end{equation}%
Utilising (\ref{e.2.3}) and (\ref{e.2.4}) we obtain (\ref{e.2.2}).
\end{proof}

\begin{lemma}
\label{l.2.3} Let $\left( V_{j}\right) _{j\in \mathbb{N}}\subset B\left(
H\right) $ and $V\in B\left( H\right) .$ If $V_{n}\rightarrow V$ in $B\left(
H\right) $ and $V_{j}V=VV_{j}$ for any $j\in \mathbb{N}$, then $%
\lim_{n\rightarrow \infty }r\left( V_{n}\right) =r\left( V\right) .$
\end{lemma}

\begin{proof}
Utilising (\ref{e.2.2}) and (\ref{e.0.2}) we have%
\begin{equation*}
\left\vert r\left( V_{n}\right) -r\left( V\right) \right\vert \leq r\left(
V_{n}-V\right) \leq \left\Vert V_{n}-V\right\Vert
\end{equation*}%
for any $n\in \mathbb{N}$ which produces the desired result.
\end{proof}

For power series $f\left( z\right) =\sum_{n=0}^{\infty }a_{n}z^{n}$ with
complex coefficients we can naturally construct another power series which
have as coefficients the absolute values of the coefficient of the original
series, namely, $f_{a}\left( z\right) :=\sum_{n=0}^{\infty }\left\vert
a_{n}\right\vert z^{n}$. It is obvious that this new power series have the
same radius of convergence as the original series, and that if all
coefficients $a_{n}\geq 0,$ then $f_{a}=f$.

We can state and prove now our first result:

\begin{theorem}
\label{t.2.1}Let $f\left( z\right) =\sum_{n=0}^{\infty }a_{n}z^{n}$ be a
power series with complex coefficients and convergent on the open disk $%
D\left( 0,R\right) \subset \mathbb{C}$, $R>0.$ If $T\in B\left( H\right) $
with $\left\Vert T\right\Vert <R,$ then%
\begin{equation}
r\left[ f\left( T\right) \right] \leq f_{a}\left( r\left( T\right) \right) .
\label{e.2.5}
\end{equation}
\end{theorem}

\begin{proof}
Let $m\geq 1$ and define $V_{j}:=a_{j}T^{j}$ for $j\in \left\{
0,...,m\right\} .$ We observe that $V_{j}V_{k}=V_{k}V_{j}$ for any $j,k\in
\left\{ 0,...,m\right\} $ and by Lemma \ref{l.2.1} we then have%
\begin{align}
r\left( \sum_{j=0}^{m}a_{j}T^{j}\right) & \leq \sum_{j=0}^{m}r\left(
a_{j}T^{j}\right) =\sum_{j=0}^{m}\left\vert a_{j}\right\vert r\left(
T^{j}\right)  \label{e.2.6} \\
& =\sum_{j=0}^{m}\left\vert a_{j}\right\vert r^{j}\left( T\right) ,  \notag
\end{align}%
where for the last equality we used the property (\ref{e.0.3}).

Now, consider the sequence $V_{m}:=\sum_{j=0}^{m}a_{j}T^{j}.$ Since $%
\left\Vert T\right\Vert <R$ it follows that $V_{m}\rightarrow f\left(
T\right) $ in $B\left( H\right) .$ Also, since $f\left( T\right) $ commutes
with each of the $a_{j}T^{j}$ it follows that $f\left( T\right) $ also
commutes with $V_{m}$ and by Lemma \ref{l.2.3} we have that 
\begin{equation*}
\lim_{m\rightarrow \infty }r\left( \sum_{j=0}^{m}a_{j}T^{j}\right) =r\left(
\sum_{j=0}^{\infty }a_{j}T^{j}\right) =r\left[ f\left( T\right) \right] .
\end{equation*}%
Therefore, by taking the limit over $m\rightarrow \infty $ and taking into
account the fact that $\sum_{j=0}^{\infty }\left\vert a_{j}\right\vert
r^{j}\left( T\right) $ is convergent since $r\left( T\right) \leq \left\Vert
T\right\Vert <R,$ we deduce the desired result (\ref{e.2.5}).
\end{proof}

\begin{corollary}
\label{c.2.1}Let $f\left( z\right) =\sum_{n=0}^{\infty }a_{n}z^{n}$ be a
power series with nonnegative coefficients and convergent on the open disk $%
D\left( 0,R\right) \subset \mathbb{C}$, $R>0.$ If $T\in B\left( H\right) $
with $\left\Vert T\right\Vert <R,$ then%
\begin{equation}
r\left[ f\left( T\right) \right] \leq f\left( r\left( T\right) \right) .
\label{e.2.7}
\end{equation}
\end{corollary}

\section{Inequalities for Two Commuting Operators}

We can consider now the case of two operators.

\begin{theorem}
\label{t.2.2}Let $f\left( z\right) =\sum_{n=0}^{\infty }a_{n}z^{n}$ be a
power series with complex coefficients and convergent on the open disk $%
D\left( 0,R\right) \subset \mathbb{C}$, $R>0.$ If $A,B\in B(H)$ are \textit{%
commutative }and for $p>1,\frac{1}{p}+\frac{1}{q}=1$ 
\begin{equation}
\left\Vert A\right\Vert ^{p},\left\Vert B\right\Vert ^{q}<R,  \label{e.2.8}
\end{equation}%
then we have%
\begin{equation}
r\left[ f\left( AB\right) \right] \leq \min \left\{ L_{1},L_{2}\right\}
\label{e.2.8.a}
\end{equation}%
where%
\begin{equation*}
L_{1}:=f_{a}^{1/p}\left( r^{p}\left( A\right) \right) f_{a}^{1/q}\left(
r^{q}\left( B\right) \right)
\end{equation*}%
and%
\begin{equation*}
L_{2}:=\frac{f_{a}\left( r^{p}\left( A\right) \right) f_{a}\left(
r^{q}\left( B\right) \right) }{f_{a}\left( r^{p-1}\left( A\right)
r^{q-1}\left( B\right) \right) }.
\end{equation*}%
In particular, if $\left\Vert A\right\Vert ^{2},\left\Vert B\right\Vert
^{2}<R,$ then%
\begin{equation}
r^{2}\left[ f\left( AB\right) \right] \leq f_{a}\left( r^{2}\left( A\right)
\right) f_{a}\left( r^{2}\left( B\right) \right) .  \label{e.2.8.b}
\end{equation}
\end{theorem}

\begin{proof}
Let $m\geq 1$ and write the inequality (\ref{e.2.6}) for $T=AB$ to get%
\begin{equation}
r\left( \sum_{j=0}^{m}a_{j}\left( AB\right) ^{j}\right) \leq
\sum_{j=0}^{m}\left\vert a_{j}\right\vert r^{j}\left( AB\right) .
\label{e.2.9}
\end{equation}%
Since $A$ and $B$ are commutative, then by (\ref{e.0.6}) we also have%
\begin{equation}
\sum_{j=0}^{m}\left\vert a_{j}\right\vert r^{j}\left( AB\right) \leq
\sum_{j=0}^{m}\left\vert a_{j}\right\vert r^{j}\left( A\right) r^{j}\left(
B\right)  \label{e.2.10}
\end{equation}%
for any $m\geq 1.$

Now, by H\"{o}lder's weighted inequality we have%
\begin{equation}
\sum_{j=0}^{m}\left\vert a_{j}\right\vert r^{j}\left( A\right) r^{j}\left(
B\right) \leq \left( \sum_{j=0}^{m}\left\vert a_{j}\right\vert r^{jp}\left(
A\right) \right) ^{1/p}\left( \sum_{j=0}^{m}\left\vert a_{j}\right\vert
r^{jq}\left( B\right) \right) ^{1/q}  \label{e.2.11}
\end{equation}%
for any $m\geq 1.$

Utilising (\ref{e.2.9})-(\ref{e.2.11}) we have%
\begin{equation}
r\left( \sum_{j=0}^{m}a_{j}\left( AB\right) ^{j}\right) \leq \left(
\sum_{j=0}^{m}\left\vert a_{j}\right\vert r^{jp}\left( A\right) \right)
^{1/p}\left( \sum_{j=0}^{m}\left\vert a_{j}\right\vert r^{jq}\left( B\right)
\right) ^{1/q}  \label{e.2.12}
\end{equation}%
for any $m\geq 1.$

From the condition (\ref{e.2.8}) we observe that the series whose partial
sums are involved in the inequality (\ref{e.2.12}) are convergent and by
letting $m\rightarrow \infty $ in (\ref{e.2.12}) we obtain the first
inequality in (\ref{e.2.8.a}).

Further, by utilizing the following H\"{o}lder's type inequality obtained by
Dragomir and S\'{a}ndor in 1990 \cite{DS} (see also \cite[Corollary 2.34]%
{SSD0}):%
\begin{equation}
\sum_{k=0}^{n}m_{k}\left\vert x_{k}\right\vert
^{p}\sum_{k=0}^{n}m_{k}\left\vert y_{k}\right\vert ^{q}\geq
\sum_{k=0}^{n}m_{k}\left\vert x_{k}y_{k}\right\vert
\sum_{k=0}^{n}m_{k}\left\vert x_{k}\right\vert ^{p-1}\left\vert
y_{k}\right\vert ^{q-1},  \label{e.2.13}
\end{equation}%
that holds for nonnegative numbers $m_{k}$ and complex numbers $x_{k},y_{k}$
where $k\in \left\{ 0,...,n\right\} ,$ we observe that the convergence of
the series $\sum_{k=0}^{\infty }m_{k}\left\vert x_{k}\right\vert ^{p}$ and $%
\sum_{k=0}^{\infty }m_{k}\left\vert y_{k}\right\vert ^{q}$ imply the
convergence of the series $\sum_{k=0}^{\infty }m_{k}\left\vert
x_{k}\right\vert ^{p-1}\left\vert y_{k}\right\vert ^{q-1}.$

On applying the inequality (\ref{e.2.13}) we also have%
\begin{equation}
\sum_{j=0}^{m}\left\vert a_{j}\right\vert r^{j}\left( A\right) r^{j}\left(
B\right) \leq \frac{\sum_{j=0}^{m}\left\vert a_{j}\right\vert r^{jp}\left(
A\right) \sum_{j=0}^{m}\left\vert a_{j}\right\vert r^{jq}\left( B\right) }{%
\sum_{j=0}^{m}\left\vert a_{j}\right\vert r^{j\left( p-1\right) }\left(
A\right) r^{j\left( q-1\right) }\left( B\right) }  \label{e.2.14}
\end{equation}%
for any $m\geq 1.$

Utilising (\ref{e.2.9})-(\ref{e.2.10}) and (\ref{e.2.14}) we get%
\begin{equation}
r\left( \sum_{j=0}^{m}a_{j}\left( AB\right) ^{j}\right) \leq \frac{%
\sum_{j=0}^{m}\left\vert a_{j}\right\vert r^{jp}\left( A\right)
\sum_{j=0}^{m}\left\vert a_{j}\right\vert r^{jq}\left( B\right) }{%
\sum_{j=0}^{m}\left\vert a_{j}\right\vert r^{j\left( p-1\right) }\left(
A\right) r^{j\left( q-1\right) }\left( B\right) }  \label{e.2.15}
\end{equation}%
for any $m\geq 1.$

Since all the series whose partial sums are involved in the inequality (\ref%
{e.2.15}) are convergent, then by letting $m\rightarrow \infty $ in (\ref%
{e.2.15}) we obtain the second part of (\ref{e.2.8.a}).
\end{proof}

\begin{remark}
\label{r.2.1}If the power series $f\left( z\right) =\sum_{n=0}^{\infty
}a_{n}z^{n}$ has nonnegative coefficients, then $f_{a}$ in the inequalities (%
\ref{e.2.8.a}) and (\ref{e.2.8.b}) may be replaced with $f.$
\end{remark}

From a different perspective, we also have

\begin{theorem}
\label{t.2.3}Let $f\left( z\right) =\sum_{n=0}^{\infty }a_{n}z^{n}$ be a
power series with complex coefficients and convergent on the open disk $%
D\left( 0,R\right) \subset \mathbb{C}$, $R>0.$ If $A,B\in B(H)$ are \textit{%
commutative }and $\left\Vert A\right\Vert ^{2},\left\Vert B\right\Vert
^{2}<R,$ then%
\begin{align}
r\left[ f\left( AB\right) \right] & \leq \frac{1}{2}\left[ f_{a}\left(
\left\Vert AB\right\Vert \right) +f_{a}\left( \left\Vert A^{2}\right\Vert
^{1/2}\left\Vert B^{2}\right\Vert ^{1/2}\right) \right]  \label{e.2.16} \\
& \leq \frac{1}{2}\left[ f_{a}\left( \left\Vert AB\right\Vert \right)
+f_{a}^{1/2}\left( \left\Vert A^{2}\right\Vert \right) f_{a}^{1/2}\left(
\left\Vert B^{2}\right\Vert \right) \right]  \notag
\end{align}%
and%
\begin{multline}
r\left[ f\left( AB\right) \right]  \label{e.2.16.a} \\
\leq \frac{1}{2}f_{a}\left( \left\Vert AB\right\Vert \right) +\frac{1}{2}%
\min \left\{ f_{a}\left( \left\Vert A\right\Vert ^{1/2}\left\Vert
AB^{2}\right\Vert ^{1/2}\right) ,f_{a}\left( \left\Vert A^{2}B\right\Vert
^{1/2}\left\Vert B\right\Vert ^{1/2}\right) \right\} \\
\leq \frac{1}{2}f_{a}\left( \left\Vert AB\right\Vert \right) +\frac{1}{2}%
\min \left\{ f_{a}^{1/2}\left( \left\Vert A\right\Vert \right)
f_{a}^{1/2}\left( \left\Vert AB^{2}\right\Vert \right) ,f_{a}^{1/2}\left(
\left\Vert A^{2}B\right\Vert \right) f_{a}^{1/2}\left( \left\Vert
B\right\Vert \right) \right\}
\end{multline}%
provided also that $\left\Vert A\right\Vert ,\left\Vert B\right\Vert <R.$
\end{theorem}

\begin{proof}
Since $A$ and $B$ commute, then $A^{j}$ and $B^{j}$ commute for each $j\in 
\mathbb{N}$ and by the inequality (\ref{e.0.8}) and the properties of norms
we have%
\begin{align}
r\left( \left( AB\right) ^{j}\right) & =r\left( A^{j}B^{j}\right) \leq \frac{%
1}{2}\left( \left\Vert A^{j}B^{j}\right\Vert +\left\Vert A^{2j}\right\Vert
^{1/2}\left\Vert B^{2j}\right\Vert ^{1/2}\right)  \label{e.2.17} \\
& \leq \frac{1}{2}\left( \left\Vert AB\right\Vert ^{j}+\left\Vert
A^{2}\right\Vert ^{j/2}\left\Vert B^{2}\right\Vert ^{j/2}\right)  \notag
\end{align}%
for each $j\in \mathbb{N}$.

If we multiply (\ref{e.2.17}) by $\left\vert a_{j}\right\vert ,$ sum over $j$
from $0$ to $m\geq 1$ and use the weighted Cauchy-Bunyakovsky-Schwarz
inequality, we have 
\begin{align}
& \sum_{j=0}^{m}\left\vert a_{j}\right\vert r\left( \left( AB\right)
^{j}\right)  \label{e.2.18} \\
& \leq \frac{1}{2}\left( \sum_{j=0}^{m}\left\vert a_{j}\right\vert
\left\Vert AB\right\Vert ^{j}+\sum_{j=0}^{m}\left\vert a_{j}\right\vert
\left\Vert A^{2}\right\Vert ^{j/2}\left\Vert B^{2}\right\Vert ^{j/2}\right) 
\notag \\
& \leq \frac{1}{2}\left( \sum_{j=0}^{m}\left\vert a_{j}\right\vert
\left\Vert AB\right\Vert ^{j}+\left( \sum_{j=0}^{m}\left\vert
a_{j}\right\vert \left\Vert A^{2}\right\Vert ^{j}\right) ^{1/2}\left(
\sum_{j=0}^{m}\left\vert a_{j}\right\vert \left\Vert B^{2}\right\Vert
^{j}\right) ^{1/2}\right)  \notag
\end{align}%
for any $m\geq 1.$

Now, utilising (\ref{e.2.9}) we can state the following string of
inequalities%
\begin{align}
& r\left( \sum_{j=0}^{m}a_{j}\left( AB\right) ^{j}\right)  \label{e.2.19} \\
& \leq \frac{1}{2}\left( \sum_{j=0}^{m}\left\vert a_{j}\right\vert
\left\Vert AB\right\Vert ^{j}+\sum_{j=0}^{m}\left\vert a_{j}\right\vert
\left\Vert A^{2}\right\Vert ^{j/2}\left\Vert B^{2}\right\Vert ^{j/2}\right) 
\notag \\
& \leq \frac{1}{2}\left( \sum_{j=0}^{m}\left\vert a_{j}\right\vert
\left\Vert AB\right\Vert ^{j}+\left( \sum_{j=0}^{m}\left\vert
a_{j}\right\vert \left\Vert A^{2}\right\Vert ^{j}\right) ^{1/2}\left(
\sum_{j=0}^{m}\left\vert a_{j}\right\vert \left\Vert B^{2}\right\Vert
^{j}\right) ^{1/2}\right)  \notag
\end{align}%
for any $m\geq 1.$

Since all the series whose partial sums are involved in the inequality (\ref%
{e.2.19}) are convergent, then by letting $m\rightarrow \infty $ in (\ref%
{e.2.19}) we obtain the desired inequality (\ref{e.2.16}).

Now, on making use of the first inequality in (\ref{e.0.10})\ we have%
\begin{align}
& r\left( A^{j}B^{j}\right)  \label{e.2.19.a} \\
& \leq \frac{1}{2}\left[ \left\Vert A^{j}B^{j}\right\Vert +\min \left\{
\left\Vert A^{j}\right\Vert ^{1/2}\left\Vert A^{j}B^{2j}\right\Vert
^{1/2},\left\Vert A^{2j}B^{j}\right\Vert ^{1/2}\left\Vert B^{j}\right\Vert
^{1/2}\right\} \right]  \notag \\
& \leq \frac{1}{2}\left[ \left\Vert AB\right\Vert ^{j}+\min \left\{
\left\Vert A\right\Vert ^{j/2}\left\Vert AB^{2}\right\Vert ^{j/2},\left\Vert
A^{2}B\right\Vert ^{j/2}\left\Vert B\right\Vert ^{j/2}\right\} \right] 
\notag
\end{align}%
for any $j\in \mathbb{N}$.

Utilising the elementary inequality for nonnegative numbers $%
p_{j},c_{j},d_{j}$ with $j\in \left\{ 0,...,m\right\} ,m\geq 1$%
\begin{equation*}
\sum_{j=0}^{m}p_{j}\min \left\{ c_{j},d_{j}\right\} \leq \min \left\{
\sum_{j=0}^{m}p_{j}c_{j},\sum_{j=0}^{m}p_{j}d_{j}\right\} ,
\end{equation*}%
we obtain from (\ref{e.2.19}) by multiplying with $\left\vert
a_{j}\right\vert $ and summing over $j\in \left\{ 0,...,m\right\} $ that%
\begin{multline}
\sum_{j=0}^{m}\left\vert a_{j}\right\vert r\left( \left( AB\right)
^{j}\right)  \label{e.2.19.b} \\
\leq \frac{1}{2}\sum_{j=0}^{m}\left\vert a_{j}\right\vert \left\Vert
AB\right\Vert ^{j}+\frac{1}{2}\min \left\{ \sum_{j=0}^{m}\left\vert
a_{j}\right\vert \left\Vert A\right\Vert ^{j/2}\left\Vert AB^{2}\right\Vert
^{j/2},\sum_{j=0}^{m}\left\vert a_{j}\right\vert \left\Vert
A^{2}B\right\Vert ^{j/2}\left\Vert B\right\Vert ^{j/2}\right\}
\end{multline}%
for any $m\geq 1.$

Following a similar argument to the one outlined above we get the first
inequality in (\ref{e.2.16.a}).

The second inequality follows by the Cauchy-Bunyakovsky-Schwarz inequality
and the details are omitted.
\end{proof}

\begin{remark}
\label{r.2.2}If $f\left( z\right) =\sum_{n=0}^{\infty }a_{n}z^{n}$ is a
power series with nonnegative coefficients then $f_{a}$ from the
inequalities (\ref{e.2.16}) and (\ref{e.2.16.a}) may be replaced with $f.$
\end{remark}

Finally, on making use of the inequality%
\begin{equation}
r\left( AB\right) \leq \frac{1}{2}\left\Vert AB\right\Vert +\frac{1}{2}%
\times \left\{ 
\begin{array}{ll}
\left\Vert A\right\Vert ^{1/2}\left\Vert B\right\Vert ^{1/2}\left\Vert
AB\right\Vert ^{1/2} &  \\ 
&  \\ 
\min \left\{ \left\Vert A\right\Vert \left\Vert B^{2}\right\Vert
^{1/2},\left\Vert A^{2}\right\Vert ^{1/2}\left\Vert B\right\Vert \right\} & 
\end{array}%
\right.  \label{e.2.20}
\end{equation}%
we can also state the result:

\begin{proposition}
\label{p.3.1}Let $f\left( z\right) =\sum_{n=0}^{\infty }a_{n}z^{n}$ be a
power series with complex coefficients and convergent on the open disk $%
D\left( 0,R\right) \subset \mathbb{C}$, $R>0.$ If $A,B\in B(H)$ are \textit{%
commutative }and $\left\Vert A\right\Vert ^{2},\left\Vert B\right\Vert
^{2}<R,$ then%
\begin{align}
& r\left[ f\left( AB\right) \right]  \label{e.2.21} \\
& \leq \frac{1}{2}f_{a}\left( \left\Vert AB\right\Vert \right)  \notag \\
& +\frac{1}{2}\times \left\{ 
\begin{array}{ll}
f_{a}\left( \left\Vert A\right\Vert ^{1/2}\left\Vert B\right\Vert
^{1/2}\left\Vert AB\right\Vert ^{1/2}\right) &  \\ 
&  \\ 
\min \left\{ f_{a}\left( \left\Vert A\right\Vert \left\Vert B^{2}\right\Vert
^{1/2}\right) ,f_{a}\left( \left\Vert A^{2}\right\Vert ^{1/2}\left\Vert
B\right\Vert \right) \right\} & 
\end{array}%
\right.  \notag \\
& \leq \frac{1}{2}f_{a}\left( \left\Vert AB\right\Vert \right)  \notag \\
& +\frac{1}{2}\times \left\{ 
\begin{array}{ll}
f_{a}^{1/2}\left( \left\Vert A\right\Vert \left\Vert B\right\Vert \right)
f_{a}^{1/2}\left( \left\Vert AB\right\Vert \right) &  \\ 
&  \\ 
\min \left\{ f_{a}^{1/2}\left( \left\Vert A\right\Vert ^{2}\right)
f_{a}^{1/2}\left( \left\Vert B^{2}\right\Vert \right) ,f_{a}^{1/2}\left(
\left\Vert A^{2}\right\Vert \right) f_{a}^{1/2}\left( \left\Vert
B\right\Vert ^{2}\right) \right\} . & 
\end{array}%
\right.  \notag
\end{align}
\end{proposition}

\section{Some Examples}

As some natural examples that are useful for applications, we can point out
that, if 
\begin{align}
f\left( z\right) & =\sum_{n=1}^{\infty }\frac{\left( -1\right) ^{n}}{n}%
z^{n}=\ln \frac{1}{1+z},\text{ }z\in D\left( 0,1\right) ;  \label{E1} \\
g\left( z\right) & =\sum_{n=0}^{\infty }\frac{\left( -1\right) ^{n}}{\left(
2n\right) !}z^{2n}=\cos z,\text{ }z\in \mathbb{C}\text{;}  \notag \\
h\left( z\right) & =\sum_{n=0}^{\infty }\frac{\left( -1\right) ^{n}}{\left(
2n+1\right) !}z^{2n+1}=\sin z,\text{ }z\in \mathbb{C}\text{;}  \notag \\
l\left( z\right) & =\sum_{n=0}^{\infty }\left( -1\right) ^{n}z^{n}=\frac{1}{%
1+z},\text{ }z\in D\left( 0,1\right) ;  \notag
\end{align}%
then the corresponding functions constructed by the use of the absolute
values of the coefficients are%
\begin{align}
f_{a}\left( z\right) & =\sum_{n=1}^{\infty }\frac{1}{n}z^{n}=\ln \frac{1}{1-z%
},\text{ }z\in D\left( 0,1\right) ;  \label{E2} \\
g_{a}\left( z\right) & =\sum_{n=0}^{\infty }\frac{1}{\left( 2n\right) !}%
z^{2n}=\cosh z,\text{ }z\in \mathbb{C}\text{;}  \notag \\
h_{A}\left( z\right) & =\sum_{n=0}^{\infty }\frac{1}{\left( 2n+1\right) !}%
z^{2n+1}=\sinh z,\text{ }z\in \mathbb{C}\text{;}  \notag \\
l_{A}\left( z\right) & =\sum_{n=0}^{\infty }z^{n}=\frac{1}{1-z},\text{ }z\in
D\left( 0,1\right) .  \notag
\end{align}%
Other important examples of functions as power series representations with
nonnegative coefficients are:%
\begin{align}
\exp \left( z\right) & =\sum_{n=0}^{\infty }\frac{1}{n!}z^{n}\qquad z\in 
\mathbb{C}\text{,}  \label{E3} \\
\frac{1}{2}\ln \left( \frac{1+z}{1-z}\right) & =\sum_{n=1}^{\infty }\frac{1}{%
2n-1}z^{2n-1},\qquad z\in D\left( 0,1\right) ;  \notag \\
\sin ^{-1}\left( z\right) & =\sum_{n=0}^{\infty }\frac{\Gamma \left( n+\frac{%
1}{2}\right) }{\sqrt{\pi }\left( 2n+1\right) n!}z^{2n+1},\qquad z\in D\left(
0,1\right) ;  \notag \\
\tanh ^{-1}\left( z\right) & =\sum_{n=1}^{\infty }\frac{1}{2n-1}%
z^{2n-1},\qquad z\in D\left( 0,1\right)  \notag \\
_{2}F_{1}\left( \alpha ,\beta ,\gamma ,z\right) & =\sum_{n=0}^{\infty }\frac{%
\Gamma \left( n+\alpha \right) \Gamma \left( n+\beta \right) \Gamma \left(
\gamma \right) }{n!\Gamma \left( \alpha \right) \Gamma \left( \beta \right)
\Gamma \left( n+\gamma \right) }z^{n},\alpha ,\beta ,\gamma >0,  \notag \\
z& \in D\left( 0,1\right) ;  \notag
\end{align}%
where $\Gamma $ is \textit{Gamma function}.

If $T\in B\left( H\right) $ with $\left\Vert T\right\Vert <1,$ then by (\ref%
{e.2.5}) we have the inequalities%
\begin{equation*}
r\left[ \left( I\pm T\right) ^{-1}\right] \leq \left[ 1-r\left( T\right) %
\right] ^{-1},
\end{equation*}%
\begin{equation*}
r\left[ \ln \left( I\pm T\right) ^{-1}\right] \leq \ln \left[ 1-r\left(
T\right) \right] ^{-1},
\end{equation*}%
\begin{equation*}
r\left[ \sin ^{-1}\left( T\right) \right] \leq \sin ^{-1}\left[ r\left(
T\right) \right]
\end{equation*}%
and%
\begin{equation*}
r\left[ _{2}F_{1}\left( \alpha ,\beta ,\gamma ,T\right) \right] \leq
_{2}F_{1}\left( \alpha ,\beta ,\gamma ,r\left( T\right) \right) .
\end{equation*}

If $T\in B\left( H\right) ,$ then by the same inequality we have%
\begin{equation*}
r\left[ \exp \left( T\right) \right] \leq \exp \left[ r\left( T\right) %
\right] ,\text{ }
\end{equation*}%
\begin{equation*}
r\left[ \sin \left( T\right) \right] ,r\left[ \sinh \left( T\right) \right]
\leq \sinh \left( r\left( T\right) \right)
\end{equation*}%
and%
\begin{equation*}
r\left[ \cos \left( T\right) \right] ,r\left[ \cosh \left( T\right) \right]
\leq \cosh \left( r\left( T\right) \right) .
\end{equation*}

If $A,B\in B(H)$ are commutative\textit{\ }and $\left\Vert A\right\Vert
,\left\Vert B\right\Vert <1,$ then by the inequality (\ref{e.2.8.a}) we have 
\begin{equation*}
r\left[ \left( I\pm AB\right) ^{-1}\right] \leq \left\{ 
\begin{array}{ll}
\left( 1-r^{p}\left( A\right) \right) ^{-1/p}\left( 1-r^{q}\left( B\right)
\right) ^{-1/q}, &  \\ 
&  \\ 
\frac{\left( 1-r^{p}\left( A\right) \right) ^{-1}\left( 1-r^{q}\left(
B\right) \right) ^{-1}}{\left( 1-r^{p-1}\left( A\right) r^{q-1}\left(
B\right) \right) ^{-1}}, & 
\end{array}%
\right.
\end{equation*}%
where $p>1,\frac{1}{p}+\frac{1}{q}=1,$ and by the inequalities (\ref{e.2.16}%
), (\ref{e.2.16.a}) and (\ref{e.2.21}) we have 
\begin{multline*}
r\left[ \left( I\pm AB\right) ^{-1}\right] \leq \frac{1}{2}\left(
1-\left\Vert AB\right\Vert \right) ^{-1} \\
+\frac{1}{2}\times \left\{ 
\begin{array}{ll}
\left( 1-\left\Vert A^{2}\right\Vert ^{1/2}\left\Vert B^{2}\right\Vert
^{1/2}\right) ^{-1}, &  \\ 
&  \\ 
\min \left\{ \left( 1-\left\Vert A\right\Vert ^{1/2}\left\Vert
AB^{2}\right\Vert ^{1/2}\right) ^{-1},\left( 1-\left\Vert A^{2}B\right\Vert
^{1/2}\left\Vert B\right\Vert ^{1/2}\right) ^{-1}\right\} , &  \\ 
&  \\ 
\left( 1-\left\Vert A\right\Vert ^{1/2}\left\Vert B\right\Vert
^{1/2}\left\Vert AB\right\Vert ^{1/2}\right) ^{-1}, &  \\ 
&  \\ 
\min \left\{ \left( 1-\left\Vert A\right\Vert \left\Vert B^{2}\right\Vert
^{1/2}\right) ^{-1},\left( 1-\left\Vert A^{2}\right\Vert ^{1/2}\left\Vert
B\right\Vert \right) ^{-1}\right\} . & 
\end{array}%
\right.
\end{multline*}

By the same inequalities, if $A,B\in B(H)$ are commutative, then 
\begin{equation*}
r\left[ \exp \left( AB\right) \right] \leq \left\{ 
\begin{array}{ll}
\exp \left[ \frac{1}{p}r^{p}\left( A\right) +\frac{1}{q}r^{q}\left( B\right) %
\right] , &  \\ 
&  \\ 
\exp \left[ r^{p}\left( A\right) +r^{q}\left( B\right) -r^{p-1}\left(
A\right) r^{q-1}\left( B\right) \right] , & 
\end{array}%
\right.
\end{equation*}%
where $p>1,\frac{1}{p}+\frac{1}{q}=1,$ and%
\begin{multline*}
r\left[ \exp \left( AB\right) \right] \leq \frac{1}{2}\exp \left( \left\Vert
AB\right\Vert \right) \\
+\frac{1}{2}\times \left\{ 
\begin{array}{ll}
\exp \left( \left\Vert A^{2}\right\Vert ^{1/2}\left\Vert B^{2}\right\Vert
^{1/2}\right) , &  \\ 
&  \\ 
\exp \left[ \min \left\{ \left\Vert A\right\Vert ^{1/2}\left\Vert
AB^{2}\right\Vert ^{1/2},\left\Vert A^{2}B\right\Vert ^{1/2}\left\Vert
B\right\Vert ^{1/2}\right\} \right] , &  \\ 
&  \\ 
\exp \left( \left\Vert A\right\Vert ^{1/2}\left\Vert B\right\Vert
^{1/2}\left\Vert AB\right\Vert ^{1/2}\right) , &  \\ 
&  \\ 
\exp \left[ \min \left\{ \left\Vert A\right\Vert \left\Vert B^{2}\right\Vert
^{1/2},\left\Vert A^{2}\right\Vert ^{1/2}\left\Vert B\right\Vert \right\} %
\right] . & 
\end{array}%
\right.
\end{multline*}

\end{document}